\documentclass[11pt,leqno]{amsart}
\usepackage{amsmath,amssymb,amsthm}
\usepackage{hyperref}
\newtheorem{theorem}{Theorem}[section]
\newtheorem{lemma}[theorem]{Lemma}
\newtheorem{question}{Question}
\newtheorem{claim}[theorem]{Claim}
\newtheorem{proposition}[theorem]{Proposition}
\newtheorem{corollary}[theorem]{Corollary}
\theoremstyle{definition}
\newtheorem{definition}[theorem]{Definition}
\newtheorem{example}[theorem]{Example}

\theoremstyle{remark}
\newtheorem{remark}[theorem]{Remark}
\numberwithin{equation}{section}

\def\fnote#1{\footnote}

\def\ignora#1{}
\def\n3#1{\left\vert  \! \left\vert \! \left\vert \, #1 \, \right\vert \!
  \right\vert \! \right\vert }

\begin{document}

\title{ Diametral diameter two properties in Banach spaces }

\author{Julio Becerra Guerrero, Gin{\'e}s L{\'o}pez-P{\'e}rez and Abraham Rueda Zoca}
\address{Universidad de Granada, Facultad de Ciencias.
Departamento de An\'{a}lisis Matem\'{a}tico, 18071-Granada
(Spain)} \email{juliobg@ugr.es, glopezp@ugr.es
arz0001@correo.ugr.es}
\thanks{The first author was partially supported by MEC (Spain) Grant MTM2014-58984-P and Junta de Andaluc\'{\i}a grants
FQM-0199, FQM-1215. The second author was partially supported by
MEC (Spain) Grant MTM2015-65020-P and Junta de Andaluc\'{\i}a Grant
FQM-185.} 

\maketitle\markboth{J. Becerra, G. L\'{o}pez and A.
Rueda}{Diametral diameter two properties in Banach spaces.}

\begin{abstract}
The aim of this note is to provide several variants of the
diameter two properties for Banach spaces. We study such properties looking for the abundance of diametral
points, which holds in the setting of Banach spaces with the Daugavet property, for example, and  we introduce the diametral diameter two  properties in Banach spaces, showing for these new properties stability results, inheritance to subspaces and characterizations in terms of finite rank projections.
\end{abstract}

\section{Introduction}

\bigskip
\par

We recall that a Banach space $X$ satisfies the strong diameter
two pro\-per\-ty (SD2P), respectively diameter two property (D2P),
slice-diameter two property (LD2P),  if every convex combination
of slices, respectively every nonempty relatively weakly open
subset, every slice,  in the unit ball of $X$ has diameter $2$.
The  weak-star slice diameter two property      ($w^*$-LD2P),
weak-star diameter two property ($w^*$-D2P) and weak-star strong
diameter two property ($w^*$-SD2P) for a dual Banach space are
defined as usual, changing slices by $w^*$-slices and weak open
subsets by $w^*$- open subsets in the unit ball.  It is known that
the above six properties are extremely different as it is proved
in \cite{avd}.

Even though diameter two property theory is a very recent topic in
geometry of Banach spaces, a lot of nice results have appeared in
the last few years (e.g. \cite{abl,avd,blr3,blro,hlp}). Moreover,
it turns out that there are quite lot of examples of Banach spaces
with such properties as infinite-dimensional $C^*$-algebras
\cite{blro}, non-reflexive $M$-embedded spaces \cite{gines} or
Daugavet spaces \cite{shi}. Last example is quite important
because Banach spaces with the Daugavet property actually satisfy
the diameter two properties in a stronger way. Indeed, as it was
pointed out in \cite{ik}, Banach spaces with Daugavet property
verify that each slice of the unit ball $S$ has diameter two and
each norm-one element of $S$ is diametral, i.e. given $x\in S\cap
S_X$ it follows that
\begin{equation}\label{defidiametral}
diam(S)=\sup\limits_{y\in S} \Vert y-x\Vert.
\end{equation}
We will say that a Banach space $X$ has the diametral local
diameter two property (DLD2P) whenever $X$ verifies the above condition.
It is known that this property is stable by taking $\ell_p$
sums \cite{ik} and that is inherited to almost isometric ideals
\cite{ahntt}. Moreover, this property is different to the Daugavet
property (see again \cite{ik}).

The aim of this note is to provide extensions of the diameter two
properties in the way exposed above and make an intensive study
of such properties. Indeed, in sections 2 and 3, we shall analyze
extensions of the D2P and SD2P, respectively,  by the existence of diametral
points. Whilst we shall define the diametral diameter two property
by the obvious generalization in view of the diametral slice
diameter two property, we provide a natural extension of the SD2P
in terms of diametrality in some different way. Given a Banach
space $X$ we will say that $X$ has the diametral strong diameter
two property (DSD2P) whenever given $C$ a convex combination of
non-empty relatively weakly open subsets  of $B_X$, $x\in C$ and
$\varepsilon\in\mathbb R^+$ we can find $y\in C$ such that $\Vert
y-x\Vert>1+\Vert x\Vert-\varepsilon$. This alternative definition
is given because a convex combination of non-empty relatively
weakly open subsets of the unit ball of a Banach space does not
have to intersect to the unit sphere and it is quite clear that
(\ref{defidiametral}) implies $\Vert x\Vert=1$. We will get some
results of stability of diametral diameter two properties in terms
of $\ell_p$ sums and inheritance to subspaces. Moreover, we will
exhibit some characterizations of such properties in terms of
finite-rank projections or weakly convergent nets. 
Finally, section 4. is devoted to exhibit some open problems and
remarks.

We shall introduce some notation. We consider real Banach spaces,
$B_X$ (resp. $S_X$) denotes the closed unit ball (resp. sphere) of
the Banach space $X$. If $Y$ is a subspace of a Banach space $X$,
$X^*$ stands for the dual space of $X$. A slice of a bounded
subset $C$ of $X$ is a set of the form
$$S(C,f,\alpha):=\{x\in C:\ f( x)>M-\alpha\},$$
where $f\in X^*$, $f\neq 0$, $M=\sup_{x\in C}f(x)$ and $\alpha>0$. If $X=Y^*$ is a dual space for some Banach space $Y$ and $C$ is a bounded subset of $X$, a $w^*$-slice of $C$ is a set of the form
$$S(C,y,\alpha):=\{f\in C:f(y)>M-\alpha\},$$
where  $y\in Y$, $y\neq 0$, $M=\sup_{f\in C}f(y)$ and $\alpha>0$. $w$ (resp. $w^*$) denotes the weak (resp. weak-star) topology of a Banach space.

It is proved in \cite[Corollary 2.2]{blr3} that a Banach space $X$
has the SD2P if, and only if, $X^*$ has an octahedral norm.

Moreover, it is proved in \cite[Theorems 3.2 and 3.4]{hlp} that a
Banach space $X$ has the LD2P (respectively the D2P) if, and only
if, $X^*$ has a locally octahedral (respectively weakly
octahedral) norm.

Let $X$ be a Banach space and $Y\subseteq X$ a closed subspace.
According to \cite{aln}, we will say that $Y$ is an \textit{almost
isometric ideal in $X$} if for each $\varepsilon>0$ and
$E\subseteq X$ a finite-dimensional subspace there exists a linear
and bounded operator $T:E\longrightarrow Y$ satisfying the
following conditions:

\begin{enumerate}
\item $T(e)=e$ for each $e\in E\cap Y$.

\item For each $e\in E$ one has
$$\frac{1}{1+\varepsilon}\Vert e\Vert\leq \Vert T(e)\Vert\leq (1+\varepsilon)\Vert e\Vert.$$
\end{enumerate}

In spite of the fact that almost isometric ideals in Banach spaces
do not have to be closed, by a perturbation argument it follows
that a non-closed subspace is an almost isometric ideal if, and only if, its
closure is also an almost isometric ideal. Hence, we will consider
only closed almost isometric ideals.

In \cite{aln} is proved that each diameter two property as well as
Daugavet property are inherited to almost isometric ideals from
the whole space. This is a consequence of the following

\begin{theorem}\label{propiai}\cite[Theorem 1.4]{aln}
Let $X$ be a Banach space and $Y\subseteq X$ an almost ideal in
$X$.
Then there exists $\varphi:Y^*\longrightarrow X^*$ a Hahn-Banach
operator such that, for each $\varepsilon>0$, for each $E\subseteq
X$ finite-dimensional subspace and each $F\subseteq Y^*$ finite
dimensional subspace, there exists $T:E\longrightarrow Y$
verifying the following:

\begin{enumerate}
\item $T(e)=e$ for each $e\in E\cap Y$.

\item For each $e\in E$ one has
$$\frac{1}{1+\varepsilon}\Vert e\Vert\leq \Vert T(e)\Vert\leq (1+\varepsilon)\Vert e\Vert.$$

\item For each $e\in E$ and $f\in F$ it follows
$$\varphi(f)(e)=f(T(e)).$$
\end{enumerate}

\end{theorem}

We shall also exhibit the following known result which will be
used several times in the following. A proof can be found in
\cite[Lemma 2.1]{ik}.

\begin{lemma}\label{lemakadets}

Let $X$ be a Banach space. Consider $x^*\in S_{X^*},
\varepsilon\in\mathbb R^+$ and $x\in S(B_X,x^*, \varepsilon)\cap
S_X$. Then, given $0<\delta<\varepsilon$, there exists $y^*\in
S_{X^*}$ such that
$$x\in S(B_X,y^*,\delta)\subseteq S(B_X,x^*,\varepsilon).$$
\end{lemma}

Similarly, a dual version of Lemma above can be statedas follows.

\begin{lemma}\label{lemadualkadets}
Let $X$ be a Banach space. Consider $x\in S_{X}, 
\varepsilon\in\mathbb R^+$ and $x^*\in
S(B_{X^*},x,\varepsilon)\cap S_{X^*}$. Then, given
$0<\delta<\varepsilon$, there exists $y\in S_{X}$ such that
$$x^*\in S(B_{X^*},y,\delta)\subseteq S(B_{X^*},x,\varepsilon).$$
\end{lemma}

\section{Diametral diameter two property and stability results}

We shall start by giving the following

\begin{definition}\label{defiDD2P}
Let $X$ be a Banach space.

We will say that $X$ has the diametral  diameter two property
(DD2P) if given $W$ a non-empty relatively weakly open subset of
$B_X$, $x\in W\cap S_X$ and $\varepsilon\in\mathbb R^+$ there
exists $y\in W$ such that
\begin{equation}\label{ecuaDD2P}
\Vert x-y\Vert>2-\varepsilon.
\end{equation}
If $X$ is a dual Banach space we will say that $X$ has the
weak-star diametral  diameter two property ($w^*$-DD2P) if given
$W$ a non-empty relatively weakly-star open subset of $B_X$, $x\in
W\cap S_X$ and $\varepsilon\in\mathbb R^+$ there exists $y\in W$
satisfying (\ref{ecuaDD2P}).
\end{definition}

From \cite[Lemma 2.3]{shi} we get that each Banach space enjoying
to have Daugavet property satisfies DD2P. However, there are
Banach spaces with the DD2P which do not enjoy to have the
Daugavet property.

\begin{example}\label{ejeDD2Pnodauga}

Let $X$ be the renorming of $C([0,1])$ given in \cite{ahntt}
satisfying that $X$ is MLUR, has the DLD2P and $X$ fails SD2P. Then $X$ fails the Daugavet property. However, $X$ has the DD2P
because $X$ has the DLD2P, applying  the well known Choquet lemma \cite[Lemma 3.40]{rusos}. 
\end{example}

It is known that a Banach space $X$ has the D2P if, and only if,
$X^{**}$ has the $w^*$-D2P. However, this fact is far from being
true for the DD2P. Indeed, applying the weak-star lower semicontinuity of a bidual norm is easy to get the following

\begin{proposition}\label{DD2Pbidual}

Let $X$ be a Banach space.
If $X^{**}$ has the $w^*$-DD2P, then $X$ has the DD2P.

\end{proposition}

\begin{remark}\label{remarkDD2Pbidual}

The converse of Proposition \ref{DD2Pbidual} is not true. Indeed
consider $X:=\mathcal C(K)$, for an infinite compact Hausdorff and perfect topological space $K$. Now $X$ has the DD2P as being a
Daugavet space. However, $B_{X^*}$ has denting points, so $X^*$
fails the DLD2P and, consequently, $X^{**}$ fails the $w^*$-DLD2P
\cite[Theorem 3.6]{ahntt}.

\end{remark}

Now we shall provide several characterizations of the DD2P. First
of all, we shall show a useful characterization of the DD2P in
terms of weakly convergent nets which will be used in order to
prove the stability of the DD2P by $\ell_p$ sums.

\begin{proposition}\label{caraDD2Predes}

Let $X$ be a Banach space. The following assertions are
equivalent:

\begin{enumerate}
\item $X$ has DD2P. \item For each $x\in S_X$ there exists a net
$\{x_s\}\subset B_X$ which converges weakly to $x$ and such that
$$\{\Vert x-x_s\Vert\}\rightarrow 2.$$
\end{enumerate}

\end{proposition}

\begin{proof}
(1)$\Rightarrow $(2). Pick $\mathcal U$ a neighborhood system of
$x$ in the weak topology relative to $B_X$. Now, for each $U\in
\mathcal U$ and every $\varepsilon\in\mathbb R^+$, choose
$x_{(U,\varepsilon)}\in U$ such that
$$\Vert x-x_{(U,\varepsilon)}\Vert\geq 2-\varepsilon.$$
Such $x_{(U,\varepsilon)}$ exists because $X$ has the DD2P. Now,
considering in $\mathcal U\times \mathbb R^+$ the partial order
given by the reverse inclusion in $\mathcal U$ and the inverse
natural order in $\mathbb R$ we conclude that
$\{x_{(U,\varepsilon)}\}_{(U,\varepsilon)\in\mathcal U\times
\mathbb R^+}\rightarrow x$ in the weak topology of $B_X$. It is
also clear that $\{\Vert
x-x_{(U,\varepsilon)}\Vert\}_{(U,\varepsilon)\in\mathcal U\times
\mathbb R^+}\rightarrow 2$.

(2)$\Rightarrow $(1). Pick $W$ a non-empty relatively weakly open
subset of $B_X$,  $x\in W\cap S_X$ and $\varepsilon\in\mathbb R^+$
and let us prove that there exists $y\in W$ such that $\Vert
x-y\Vert>2-\varepsilon$. By assumption there exists $\{x_s\}$ a
net in $B_X$ such that
$$\Vert x-x_s\Vert\rightarrow 2,$$
and
$$\{x_s\}\stackrel{w}{\rightarrow}x.$$
From both convergences then there exists $s$ such that $x_s\in W$
and $\Vert x-x_s\Vert>2-\varepsilon$. Now (1) follows choosing
$y:=x_s$.
\end{proof}

Now, for dual Banach spaces we have the following characterization
of the $w^*$-DD2P, as the above one.

\begin{corollary}\label{caraw*DD2Predes}
Let $X$ be a dual Banach space. The following assertions are
equivalent:

\begin{enumerate}
\item $X$ has $w^*$-DD2P. \item For each $x\in S_X$ there exists a
net $\{x_s\}$ in $B_X$ which converges to $x$ in the weak-star
topology such that
$$\{\Vert x-x_s\Vert\}\rightarrow 2.$$
\end{enumerate}

\end{corollary}

\begin{remark}\label{caradaugaredes}

In view of Proposition \ref{caraDD2Predes}, Daugavet property can
also be easily  characterized in terms of weakly convergent nets.
Indeed it is straighforward to prove from \cite[Lemma 2.3]{shi}
that a Banach space $X$ has the Daugavet property if, and only if,
given $x,y\in S_X$ there exists $\{y_s\}$ a net in $B_X$ weakly
convergent to $y$ such that
$$\{\Vert x-y_s\Vert\}\rightarrow 2.$$
\end{remark}

In \cite{ik} it is proved a characterization of DLD2P in terms of
the behavior of rank-one projections in a Banach space. It turns
out to be also true that DD2P can be characterized regarding the
behaviour of the rank-one projections. In fact, we have the
following characterization of the DD2P.

\begin{proposition}
Let $X$ be a Banach space. The following assertions are
equivalent:

\begin{enumerate}
\item $X$ has the DD2P.

\item For each $x_1^*,\ldots, x_n^*\in S_{X^*}$ and $x\in X$ such
that $x_i^*(x)\neq 0$, if we define
$$p_i:=x_i^*\otimes \frac{x}{x_i^*(x)}\ \forall i\in\{1,\ldots, n\}$$
one has that, for each $\varepsilon\in\mathbb R^+$, there exists
$y\in B_X$ such that
$$\Vert y-p_i(y)\Vert>2-\varepsilon\ \forall i\in\{1,\ldots, n\}$$
and
$$\frac{x_i^*(y)}{x_i^*(x)}\geq 0\ \forall i\in\{1,\ldots, n\}$$

\item Given $S:=S(B_X,x,\delta)$ a weak-star slice of $B_{X^*}$
and $x_1^*,\ldots, x_n^*\in S\cap S_{X^*}$ there exist $y^*\in S$
and $y\in S_{X}$ such that
$$(x_i^*-y^*)(y)>2-\delta\ \forall i\in\{1,\ldots, n\}.$$
\end{enumerate}
\end{proposition}

\begin{proof}
(1)$\Rightarrow$(2).

Consider $x_1^*,\ldots, x_n^*\in S_X$, $x\in X$ and
$p_i:=x^*_i\otimes \frac{x}{x_i^*(x)}$ for each $i\in\{1,\ldots,
n\}$.

Consider $\varepsilon>0$ such that $\varepsilon<2$. Note that
$$\frac{x}{\Vert x\Vert}\in W:=\left\{y\in B_X\ : \left\vert \frac{x_i^*(y)}{x_i^*(x)}-\frac{1}{\Vert x\Vert} \right\vert\Vert x\Vert<\frac{\varepsilon}{2} \right\},$$
where $W$ is a relatively weakly open subset of $B_X$. Moreover
$\frac{x}{\Vert x\Vert}\in S_X$. As $X$ has the DD2P we can assure 
the existence of an element $y\in W$ such that $\left\Vert
y-\frac{x}{\Vert x\Vert}\right\Vert>2-\frac{\varepsilon}{2}$.

Now, on the one hand, as $y\in W$, given $i\in\{1,\ldots, n\}$,
one has
$$\left\vert \frac{x_i^*(y)}{x_i^*(x)}-\frac{1}{\Vert x\Vert} \right\vert\Vert x\Vert <\frac{\varepsilon}{2}\Rightarrow  \frac{x_i^*(y)}{x_i^*(x)}>\frac{1}{\Vert x\Vert}-\frac{\varepsilon}{2\Vert x\Vert}=\frac{1-\frac{\varepsilon}{2}}{\Vert x\Vert}\geq 0.$$

On the other hand, given $i\in\{1,\ldots, n\}$, it follows
$$\Vert y-p_i(y)\Vert\geq \left\Vert y-\frac{x}{\Vert x\Vert}\right\Vert-\left\Vert\frac{x}{\Vert x\Vert}-p_i(y)\right\Vert>2-\frac{\varepsilon}{2}-\left\Vert \frac{x}{\Vert x\Vert}-\frac{x_i^*(y)}{x_i^*(x)}x\right\Vert>2-\varepsilon $$
since $y\in W$. So (2) follows.

(2)$\Rightarrow$(1).

Consider $W:=\bigcap\limits_{i=1}^n S(B_X,y_i^*,\varepsilon_i)$ a
non-empty relatively weakly open subset of $B_X$ and pick $x\in
W\cap S_X$. In order to prove that $X$ has the DD2P choose
$0<\varepsilon<\min\limits_{1\leq i\leq n}\varepsilon_i$. By Lemma
\ref{lemakadets} we can find, for each $i\in\{1,\ldots, n\}$, a
functional $x_i^*\in S_{X^*}$ such that
$$x\in S(B_X,x_i^*,\varepsilon)\subseteq S(B_X,y_i^*,\varepsilon_i)\ \forall i\in\{1,\ldots, n\},$$
and so $x\in \bigcap\limits_{i=1}^n
S(B_X,x_i^*,\varepsilon)\subseteq W$. Consider $\eta\in\mathbb
R^+$ small enough to satisfy $x_i^*(x)(1-\eta)>1-\varepsilon$ for
each $i\in\{1,\ldots, n\}$. For each $i\in\{1,\ldots, n\}$ define
$$p_i:=x_i^*\otimes \frac{x}{x_i^*(x)}.$$
From the hypothesis we can find $y\in B_X$ such that
$$\Vert y-p_i(y)\Vert>2-\eta\ \forall i\in\{1,\ldots, n\}$$
and
$$\frac{x_i^*(y)}{x_i^*(x)}\geq 0.$$
Now, on the one hand, one has
$$1-\eta<\Vert p_i(y)\Vert=\left\Vert \frac{x_i^*(y)}{x_i^*(x)}x\right\Vert=\frac{x_i^*(y)}{x_i^*(x)}\ \forall i\in\{1,\ldots, n\}.$$
So $x_i^*(y)>(1-\eta)x_i^*(x)>1-\varepsilon$ for each
$i\in\{1,\ldots, n\}$ and, consequently, $y\in
\bigcap\limits_{i=1}^n S(B_X,x_i^*,\varepsilon)\subseteq W$.
Moreover, chosen $i\in\{1,\ldots, n\}$, it follows
$$\Vert y-x\Vert\geq \Vert y-p_i(y)\Vert-\Vert p_i(y)-x\Vert>2-\eta-\left\vert \frac{x_i^*(y)}{x_i^*(x)}-1 \right\vert=2-\eta-\vert x_i^*(y)-x_i^*(x)\vert x_i^*(x)$$
$$>2-\eta-\varepsilon.$$
As $0<\varepsilon<\min\limits_{1\leq i\leq n}\varepsilon_i$ was
arbitrary we conclude the desired result.

(1)$\Rightarrow$(3).  Let $S$ and $x_1^*,\ldots, x_n^*$ be as in
the hypothesis and pick $0<\eta<\delta$. Now given
$i\in\{1,\ldots, n\}$ one has
$$x_i^*\in S\Leftrightarrow x_i^*(x)>1-\delta\Leftrightarrow x\in S(B_{X},x_i^*,\delta).$$
So $x\in \bigcap\limits_{i=1}^n S(B_{X},x_i^*,\delta)\cap S_X$. As
$X$ has the DD2P then there exists $y\in \bigcap\limits_{i=1}^n
S(B_{X},x_i^*,\delta)\cap S_X$ such that
$$\Vert x-y\Vert>2-\eta\Rightarrow \ \exists\ y^*\in S_{X^*}\ /\ y^*(x)-y^*(y)>2-\eta\Rightarrow \left\{\begin{array}{cc}
y^*(x)>1-\eta\\
y^*(y)<-1+\eta
\end{array} \right.  .$$
So $y^*(x)>1-\delta$ and thus $y^*\in S$. In addition, given
$i\in\{1,\ldots, n\}$, it follows
$$ (x_i^*-y^*)(y)=x_i^*(y)-y^*(y)>1-\delta+1-\eta=2-\delta-\eta.$$
From the arbitrariness of $0<\eta<\delta$ we get the desired
result by a perturbation argument, if necessary.

(3)$\Rightarrow$(1). Let $W:=\bigcap\limits_{i=1}^n
S(B_{X},y_i^*,\varepsilon_i)$ be a non-empty relatively weakly
open subset of $B_{X}$ and consider $x\in W\cap S_{X}$. Pick
$0<\delta<\min\limits_{1\leq i\leq n} \varepsilon_i$. From Lemma
\ref{lemakadets} we can find, for each $i\in\{1,\ldots, n\}$, an
element $x_i^*\in S_X$ such that
$$x\in S(B_{X},x_i^*,\delta)\subseteq S(B_{X},y_i^*,\varepsilon_i)$$
holds for each $i\in\{1,\ldots, n\}$. Now $x_1^*,\ldots, x_n^*\in
S(B_{X^*},x,\delta)$. From assumptions we can find $y^*\in
S(B_{X^*},x,\delta)$ and $y\in S_{X}$ such that
$$(x_i^*-y^*)(y)>2-\delta$$
holds for each $i\in\{1,\ldots, n\}$. Now, on the one hand
$$x_i^*(y)>1-\delta\Rightarrow y\in\bigcap\limits_{i=1}^n S(B_{X},x_i^*,\delta)\subseteq W.$$
Moreover, as $y^*\in S(B_{X^*},x,\delta)$, it follows
$$\Vert x-y\Vert\geq y^*(x)-y^*(y)>1-\delta+1-\delta=2(1-\delta).$$
From the arbitrariness of $0<\delta<\min\limits_{1\leq i\leq
n}\varepsilon_i$ we have that $X$ has the DD2P, as
desired.
\end{proof}

\begin{remark}
Note that given $p_1,\ldots, p_n$ rank one projections as in above Proposition one has
$$\Vert I-p_i\Vert\geq 2$$
whenever $X$ enjoys to have the DLD2P. However, if $X$ also
satisfies the DD2P these projections can be ``normed'' by a common
point of the space.
\end{remark}

A dual version of above Proposition is the following

\begin{proposition}\label{caraDD2Pdual}

Let $X$ be a Banach space. The following assertions are
equivalent:

\begin{enumerate}
\item $X^*$ has the $w^*$-DD2P.

\item For each $x_1,\ldots, x_n\in S_{X}$ and $x^*\in X^*$ such
that $x^*(x_i)\neq 0$, if we define
$$p_i:=\frac{x^*}{x^*(x_i)}\otimes x_i \ \forall i\in\{1,\ldots, n\}$$
one has that, for each $\varepsilon\in\mathbb R^+$, there exists
$y^*\in B_{X^*}$ such that
$$\Vert y^*-p_i(y^*)\Vert>2-\varepsilon\ \forall i\in\{1,\ldots, n\}$$
and
$$\frac{y^*(x_i)}{x^*(x_i)}\geq 0\ \forall i\in\{1,\ldots, n\}$$

\item Given $S:=S(B_X,x^*,\delta)$ a slice of $B_{X}$ and
$x_1,\ldots, x_n\in S\cap S_{X}$ there exist $y\in S$ and $y^*\in
S_{X^*}$ such that
$$y^*(x_i-y)>2-\delta\ \forall i\in\{1,\ldots, n\}.$$
\end{enumerate}
\end{proposition}

It is known that DLD2P is stable under taking $\ell_p$-sums.
Indeed, given two Banach spaces $X$ and $Y$ and $1\leq p\leq
\infty$, the Banach space $X\oplus_p Y$ has the DLD2P if, and only
if, $X$ and $Y$ enjoy to have the DLD2P \cite[Theorem 3.2]{ik}.

Our aim is to establish the same result for the DD2P. We shall
begin with the stability result

\begin{theorem}\label{DD2Pestableprodu}
Let $X,Y$ be Banach spaces which satisfy the DD2P and let $1\leq
p\leq \infty$. Then $X\oplus_p Y$ enjoys to have the DD2P.
\end{theorem}

\begin{proof}

Define $Z:=X\oplus_p Y$ , pick $(x_0,y_0)\in  S_Z$ and let us
apply Proposition \ref{caraDD2Predes}.

On the one hand, if $p=\infty$, then either $\Vert x_0\Vert=1$ or
$\Vert y_0\Vert=1$. Assume, with no loss of generality, that
$\Vert x_0\Vert=1$. As $X$ has the DD2P then there exists
$\{x_s\}$ a net in $B_X$ such that
$$\{x_s\}\rightarrow x_0$$
in the weak topology of $X$ and
$$\Vert x-x_s\Vert\rightarrow 2.$$
Then we have that
$$\{(x_s,y_0)\}\rightarrow (x_0,y_0)$$
in the weak topology of $B_Z$ (note that, from the definition of
the norm on $Z$ we have that each term of the above net  belongs to
$B_Z$). In addition, given $s$ one has
$$2\geq \Vert (x_0,y_0)-(x_s,y_0)\Vert_\infty=\max\{\Vert x-x_s\Vert,\Vert y_0\Vert\}\geq \Vert x-x_s\Vert.$$
As $\{\Vert x-x_s\Vert\}\rightarrow 2$ we conclude that $\{\Vert
(x_0,y_0)-(x_s,y_s)\Vert\}\rightarrow 2$.

On the other hand, assume $p<\infty$. As $(x_0,y_0)\in S_Z$ we
have that
$$\left(\Vert x_0\Vert^p+\Vert y_0\Vert^p\right)^\frac{1}{p}=1.$$
Now $x_0$ is an element of $\Vert x_0\Vert S_X$. As $X$ has the
DD2P then by Proposition \ref{caraDD2Predes} there exists
$\{x_s\}_{s\in S}$ a net in $\Vert x_0\Vert B_X$ such that
$$\{x_s\}\rightarrow x_0$$
in the weak topology of $X$ and
$$\{\Vert x_0-x_s\Vert\}\rightarrow 2\Vert x_0\Vert.$$
In addition, as $Y$ also has the DD2P, then there exists a net
$\{y_t\}_{t\in T}$ in $\Vert y_0\Vert B_Y$ such that
$$\{y_t\}_{t\in T}\rightarrow y_0$$
in the weak topology of $Y$ and such that
$$\{\Vert y_0-y_t\Vert\}_{t\in T}\rightarrow 2\Vert y_0\Vert.$$
Now we have $\{(x_s,y_t)\}_{(s,t)\in S\times T}\rightarrow
(x_0,y_0)$ in the weak topology of $Z$. Moreover, given $s\in S,
t\in T$ one has
$$\Vert (x_s,y_t)\Vert_p=\left(\Vert x_s\Vert^p+\Vert y_t\Vert^p\right)^\frac{1}{p}\leq \left(\Vert x_0\Vert^p+\Vert y_0\Vert^p\right)^\frac{1}{p}=1,$$
so $(x_s,y_t)\in B_Z$ for each $s\in S, t\in T$. Finally, given
$s\in S, t\in T$ it follows
$$\Vert (x_0,y_0)-(x_s,y_t)\Vert_p=\left(\Vert x_0-x_s\Vert^p+\Vert y_0-y_t\Vert^p\right)^\frac{1}{p}\rightarrow \left((2\Vert x_0\Vert)^p+(2\Vert y_0\Vert)^p\right)^\frac{1}{p}=2.$$
\end{proof}

Now let us prove the converse of the above  result.

\begin{proposition}

Let $X,Y$ be Banach space and define $Z:=X\oplus_p Y$ for $1\leq
p\leq \infty$. If $X$ fails to have DD2P so does $Z$.

\end{proposition}

\begin{proof}
As $X$ fails the DD2P then there exists $U$ a non-empty relatively
weakly open subset of $B_X$, $x_0\in U\cap S_X$ and
$\varepsilon_0\in\mathbb R^+$ such that
$$\Vert x_0-y\Vert\leq 2-\varepsilon_0\ \forall y\in U.$$
Now we shall argue by cases:

\begin{enumerate}
\item If $p=\infty$ define the weak open subset of $B_Z$ given by
$$W:=\{(x,y)\in B_Z\ :\ x\in U\cap B_X\},$$
and pick $(x_0,0)\in W$. Then for each $(x,y)\in W$ one has
$$\Vert (x_0,0)-(x,y)\Vert=\max\{\Vert x_0-x\Vert,\Vert y\Vert\}\leq \max\{2-\varepsilon_0,1\}<2,$$
as $x\in U\cap B_X$.

\item If $p<\infty$, given $\varepsilon\in\mathbb R^+$, there
exists $\delta>0$ such that
\begin{equation}\label{convenormap}
\left. \begin{array}{c}
1-\delta<\vert r\vert\leq 1, \vert s\vert\leq 1,\\
(\vert r\vert^p+\vert s\vert^p)^\frac{1}{p}\leq 1
\end{array} \right\}\Rightarrow \vert s\vert^p<\varepsilon.
\end{equation}
Define
$$W:=\{(x,y)\in B_Z\ :\ x\in U\cap B_X\ \mbox{and}\ \Vert x\Vert>1-\delta\},$$
which is a weakly open subset of $B_Z$ from the lower weakly
semicontinuity of the norm on $X$. Consider $(x_0,0)\in W$. Now,
given $(x,y)\in W$ we have from (\ref{convenormap}) that $\Vert
y\Vert^p\leq \varepsilon$. In addition, as $x\in U\cap B_X$ we
conclude $\Vert x-x_0\Vert\leq 2-\varepsilon_0$. Hence
$$\Vert (x_0,0)-(x,y)\Vert=\left(\Vert x-x_0\Vert^p+\Vert y\Vert^p \right)^\frac{1}{p}\leq \left((2-\varepsilon_0)^p+\varepsilon
\right)^\frac{1}{p}.$$
So, taking $\varepsilon$ small enough, we conclude that
$\sup\limits_{(x,y)\in W} \Vert (x_0,0)-(x,y)\Vert<2$, so we are
done.
\end{enumerate}
\end{proof}

Even though Example \ref{ejeDD2Pnodauga} shows that Daugavet
property and DD2P are different, above results provide us more
examples of such Banach spaces. Indeed, given $1<p<\infty$,
$Z:=X\oplus_p Y$ has the DD2P and fails to have Daugavet property
(actually, $Z$ fails to have the strong diameter two property
\cite[Theorem 3.2]{abl}) whenever $X,Y$ are Banach spaces with the
DD2P (in particular, Daugavet spaces).

Finally we shall study the following problem: when a subspace of a
Banach space having the DD2P inherits DD2P? In order to give a
partial answer, it has been recently proved in \cite{blr} that D2P
is hereditary to finite-codimensional subspaces. Bearing in mind
the ideas of the proof of that result, we can prove the
following

\begin{theorem}\label{pasoDD2P}
Let $X$ be a Banach space which satisfies the DD2P. If $Y$ is a closed subspace of $X$ such that $X/Y$ is
finite-dimensional then $Y$ has the DD2P.
\end{theorem}

\begin{proof}

Consider
$$W:=\{y\in Y\ :\ \vert y_i^*(y-y_0)\vert<\varepsilon_i\ \forall i\in\{1,\ldots, n\}\}, $$
for $n\in\mathbb N, \varepsilon_i\in\mathbb R^+, y_i^*\in Y^*$ for
each $i\in\{1,\ldots,n\}$ and $y_0\in Y$ such that
$$W\cap B_Y\neq \emptyset.$$
Pick $y\in W\cap S_Y$ and let us find, for each $\delta\in\mathbb
R^+$, a point  $z\in W\cap B_Y$ such that $\Vert
y-z\Vert>2-\delta$. To this aim pick an arbitrary
$\delta\in\mathbb R^+$.
Assume that $y_i^*\in X^*$ for each  $i\in\{1,\ldots, n\}$. Observe that there is no 
 loss of generality from the  Hahn-Banach  theorem.

Define
$$U:=\{x\in X\ :\ \vert y_i^*(x-y_0)\vert<\varepsilon_i\ \forall i\in\{1,\ldots, n\}\},$$
which is a weakly open set in $X$ such that $U\cap B_X\neq
\emptyset$.

Let $p:X\longrightarrow X/Y$ be the quotient map, which is a $w-w$
open map. Then  $p(U)$ is a weakly open set in  $X/Y$. In addition
$$\emptyset\neq p(U\cap B_X)\subseteq p(U)\cap p(B_X)\subseteq p(U)\cap B_{X/Y}.$$
Defining $A:=p(U)\cap B_{X/Y}$, then $A$ is a non-empty relatively
weakly open and convex subset of $B_{X/Y}$ which contains to zero.
Hence, as $X/Y$ is finite-dimensional, we can find a weakly open
set $V$ of $X/Y$, in fact a ball centered at $0$, such that
$V\subset A$ and that
\begin{equation}\label{diamabicocie}
diam(V\cap p(U)\cap B_{X/Y})=  diam(V)<\frac{\delta}{16}.
\end{equation}
As $V\subset A$ then $B:=p^{-1}(V)\cap U\cap B_X\neq \emptyset$.
Hence $B$ is a non-empty relatively weakly open subset of $B_X$.
Moreover $y\in p^{-1}(V)$ because $p(y)=0\in V$, so $y\in B\cap
S_X$. Using that $X$ satisfies the DD2P we can assure the
existence of  $v\in B$ such that
\begin{equation}\label{estimabiespad2p}
\Vert v-y\Vert>2-\frac{\delta}{16}.
\end{equation}
Note that  $v\in B$ implies $p(v)\in V= V\cap P(U)\cap B_{X/Y}$.
In view of  (\ref{diamabicocie}) it follows
$$\Vert p(v)\Vert\leq  diam(V\cap p(U)\cap B_{X/Y})<\frac{\delta}{16}.$$
Hence there exists $u\in Y$ such that $ \Vert
u-v\Vert<\frac{\delta}{16}$ and so $\Vert
u\Vert<1+\frac{\delta}{16}$. Letting  $z=\frac{u}{\Vert u\Vert}$, we
have that
$$\Vert v-z\Vert\leq \Vert u-v\Vert+\left\Vert u-\frac{u}{\Vert u\Vert}\right\Vert<\frac{\delta}{16}+\Vert u\Vert(\Vert u\Vert-1)<\frac{\delta}{16}+\left(1+\frac{\delta}{16}\right)\frac{\delta}{16}=$$
$$=\frac{\delta}{16}\left(2+\frac{\delta}{16}\right).$$
So
\begin{equation}\label{elemenposd2p}
\Vert v-z\Vert<\frac{\delta}{4}.
\end{equation}
Note that given $i\in\{1,\ldots, n\}$ and   bearing in mind
(\ref{elemenposd2p}) one has
$$\vert y_i^*(z-y_0)\vert\leq \vert y_i^*(z-v)\vert+\vert y_i^*(v-y_0)\vert\leq \Vert y_i^*\Vert\frac{\delta}{4}+\varepsilon_i,$$
using that $v\in U$. Thus, if we define
$$W_\delta:=\left\{y\in Y\ :\ \vert y_i^*(y-y_0)\vert<\varepsilon_i+\Vert y_i^*\Vert\frac{\delta}{4}\ \forall i\in\{1,\ldots, n\}\right\}$$
it follows that $v\in W_\delta\cap B_Y$. On the other hand, in
view of  (\ref{estimabiespad2p}) and (\ref{elemenposd2p}) we can
estimate
$$
\Vert y-z\Vert\geq \Vert y-v\Vert-\Vert v-z\Vert >
2-\frac{\delta}{16}-\frac{\delta}{4}>2-\delta.
$$
From here we can conclude the desired result. Indeed, for each
$i\in\{1,\ldots, n\}$ we can find
$\widehat{\varepsilon}_i\in\mathbb{R}^+$ and
$\delta_0\in\mathbb{R}^+$ such that
$$\widehat{\varepsilon}_i+\delta_0\Vert y_i^*\Vert<
\varepsilon_i\ \forall i\in\{1,\ldots, n\},$$
and that
$$y\in \widehat{W}:=\{z\in Y\ :\ \vert y_i^*(z-y_0)\vert<\widehat{\varepsilon}_i\ \forall i\in\{1,\ldots, n\}\}.$$
For $0<\delta<\delta_0$ one has
$$\widehat{W}_\delta:=\left\{y\in Y\ :\ \vert y_i^*(y-y_0)\vert<\widehat{\varepsilon}_i+\Vert y_i^*\Vert\frac{\delta}{4}\ \forall i\in\{1,\ldots, n\}\right\}\subseteq W.$$
The arbitrariness of $\delta$ in the above argument allow us to
conclude the desired result.
\end{proof}

As it is done in \cite{blr} for the $w^*$-D2P, we can conclude a
stability result for the $w^*$-DD2P.

\begin{corollary}\label{subespciosw*DD2P}

Let $X$ be a Banach space and let $Y\subseteq X$ a closed
subspace. If $X^*$ has the $w^*$-DD2P and $Y$ is finite-dimensional, then
$(X/Y)^*$ has the $w^*$-DD2P.

\end{corollary}

\begin{proof}
Consider $W$ a weakly-star open subset of $Y^\circ=(X/Y)^*$ such
that
$$W\cap B_{Y^\circ}\neq \emptyset,$$
and pick $z^*\in W\cap S_{Y^\circ}$. Now we can extend $W$ to a
weak-star open subset of $X^*$, say $U$, as it is done in Theorem
\ref{pasoDD2P} satisfying $z^*\in U\cap S_{X^*}$.

Let $p:X^*\longrightarrow X^*/Y^\circ$ be the quotient map, which
is a  $w^*-w^*$ open map. Then $p(U)$ is a weakly-star open set of
$X^*/Y^\circ$ which meets with $B_{X^*/Y^\circ}$.

If we define $A:=p(U)\cap B_{X^*/Y^\circ}$, then we have that $A$
is a relatively weak-star open and  convex subset of
$B_{X^*/Y^\circ}$ which contains to zero.

As $X^*/Y^\circ=Y^*$ is finite-dimensional, we can find $V$ a
weak-star open set of $X^*/Y^\circ$, in fact a ball centered at
zero, such that $V\subset A$ and whose diameter is as closed to
zero as desired.

From here, it is straightforward to check that  computations of
Theorem \ref{pasoDD2P} work and allow us to conclude that
$$\sup\limits_{x^*\in W\cap B_{Y^\circ}}\Vert z^*-x^*\Vert=2,$$
so $Y^\circ=(X/Y)^*$ has the $w^*$-DD2P as desired.
\end{proof}

As we have pointed out in the Introduction, the D2P is inherited
to almost isometric ideals from the whole space \cite[Proposition
3.2]{aln}. Now, following similar ideas, we get the following

\begin{proposition}

Let $X$ be a Banach space and let $Y\subseteq X$ a closed almost
isometric ideal. If $X$ has the DD2P, so does $Y$.

\end{proposition}

\begin{proof}
Take $n=1$ in the proof of Proposition \ref{aidsd2p}
\end{proof}

\section{Diametral strong diameter two property  and stability results}

Now we shall introduce the natural extension of the SD2P in the
same way the DD2P is defined.

\begin{definition}\label{defiDSD2P}

Let $X$ be a Banach space. We will say that $X$ has the diametral strong diameter two
property (DSD2P) if given $C$ a convex combination of non-empty
relatively weakly open subsets of $B_X$, $x\in C$ and
$\varepsilon\in\mathbb R^+$ then there exists $y\in C$ such that
\begin{equation}\label{ecuastrong+}
\Vert x-y\Vert>1+\Vert x\Vert-\varepsilon.
\end{equation}
If $X$ is a dual space, we will say that $X$ has the weak-star
diametral strong diameter two property ($w^*$-DSD2P) if given $C$
a convex combination of non-empty relatively weakly-star open subsets
of $B_X$, $x\in C$ and $\varepsilon\in\mathbb R^+$ then there
exists $y\in C$ satisfying (\ref{ecuastrong+}).
\end{definition}

\begin{remark}

On the one hand, note that the above definition extends the strong diameter two property
from the Bourgain lemma \cite{ggms}.

On the other hand, the condition 
(\ref{defidiametral}) is replaced with (\ref{ecuastrong+}) to get the implication DSD2P$\Rightarrow$SD2P. Indeed, consider $X$ the Banach space of Example \ref{ejeDD2Pnodauga} and $C:=\sum_{i=1}^n \lambda_i W_i$ a convex combination of non-empty relatively weakly open subsets of $B_X$. If $C\cap S_X\neq \emptyset$ then there exists $x:=\sum_{i=1}^n \lambda_i x_i\in C\cap S_X$. As $X$ is a strictly convex space we conclude that $x_1=x_2=\ldots=x_n$. Consequently  $x\in\bigcap\limits_{i=1}^n W_i\subseteq C$ and, as $X$ has the DD2P, we can find, for each $\varepsilon>0$, an element  $y\in\bigcap\limits_{i=1}^n W_i\subseteq X$ such that $\Vert y-x\Vert>2-\varepsilon$. However, $X$ fails to have the SD2P.
\end{remark}

As in the DD2P, the first example of Banach space with the DSD2P
comes from Daugavet spaces.

\begin{example}

Daugavet Banach spaces enjoy to have DSD2P.

\end{example}

\begin{proof}

Consider $X$ to be a Banach space enjoying to have the Daugavet
property. From the proof of \cite[Lemma 2.3]{shi}  it follows that
given $C$ a convex combination of non-empty relatively weakly open
subsets of $B_X$, $x\in S_X$ and $\varepsilon\in\mathbb R^+$ we
can find $y\in C$ such that
$$\Vert x+y\Vert>2-\varepsilon.$$
From here let us prove that $X$ enjoys to have the DSD2P. To this
aim pick $C:=\sum_{i=1}^n \lambda_i W_i$ a convex combination of
non-empty relatively weakly open subsets of $B_X$. Let $x\in C$
such that $x\neq 0$. From Daugavet property we can find $y\in
\sum_{i=1}^n \lambda_i (-W_i)$ such that
$$\left\Vert \frac{x}{\Vert x\Vert}+y\right\Vert>2-\varepsilon.$$
Now $-y\in C$. Moreover
$$\Vert x-(-y)\Vert\geq \left\Vert \frac{x}{\Vert x\Vert}+y\right\Vert -\left\Vert \frac{x}{\Vert x\Vert}-x\right\Vert>2-\varepsilon- \Vert x\Vert\left\vert \frac{1}{\Vert x\Vert}-1\right\vert=$$
$$2-\varepsilon- \vert 1-\Vert x\Vert\vert=2-\varepsilon-1+\Vert x\Vert=1+\Vert x\Vert-\varepsilon.$$
In order to conclude the proof assume that $0\in C$. As
$diam(C)=2$ (see the proof of \cite[Lemma 2.3]{blr3}) we can find
$x,y\in C$ such that
$$\left.\begin{array}{c}
\Vert x-y\Vert>2-\varepsilon\\
\Vert x\Vert\leq 1
\end{array}\right\}\Rightarrow \Vert y-0\Vert=\Vert y\Vert>1-\varepsilon=1+\Vert 0\Vert-\varepsilon.$$
From the arbitrariness of $C$ we conclude that $X$ has the DSD2P.

\end{proof}

Given a Banach space $X$, it is true that $X$ has the DSD2P
whenever $X^{**}$ has the $w^*$-DSD2P by a similar argument to the
one given in Proposition \ref{DD2Pbidual}. Again, the converse is
not true, because the example exhibited in Remark
\ref{remarkDD2Pbidual} also works for the DSD2P.

Moreover, DSD2P admits a characterization in terms of weakly
convergent nets as DD2P does. Indeed, we have the following

\begin{proposition}\label{caraDSD2Predes}

Let $X$ be a Banach space. The following assertions are
equivalent:

\begin{enumerate}
\item $X$ has the DSD2P. \item For each $x_1,\ldots, x_n\in B_X$
and each $\lambda_1,\ldots, \lambda_n\in\mathbb R^+$ such that
$\sum_{i=1}^n \lambda_i=1$ it follows that, for each
$i\in\{1,\ldots, n\}$, there exists $\{x_s^i\}_{s\in S}$ a net in
$B_X$ weakly convergent to $x_i$ such that
$$\left\{\left\Vert \sum_{i=1}^n \lambda_i(x_i-x_s^i) \right\Vert \right\}\rightarrow 1+\left\Vert \sum_{i=1}^n \lambda_i x_i\right\Vert.$$
\end{enumerate}

\end{proposition}

\begin{proof}

(1)$\Rightarrow$(2). Pick $\mathcal U$ a system of neighborhoods of $0$. Now, for each $U\in \mathcal U$ and $\varepsilon\in\mathbb
R^+$, pick $x_{U,\varepsilon}^i$ for each $i\in\{1,\ldots, n\}$
such that
$$x_{U,\varepsilon}^i\in (x_i+U)\cap B_X$$
and
$$\left\Vert \sum_{i=1}^n \lambda_i (x_i-x_{U,\varepsilon}^i)\right\Vert>1+\left\Vert \sum_{i=1}^n \lambda_i x_i\right\Vert-\varepsilon,$$
which can be done because $X$ has the DSD2P.

Now it is quite clear that, given $i\in\{1,\ldots, n\}$, then
$$\{x_{U,\varepsilon}\}_{(U,\varepsilon)\in\mathcal U\times \mathbb R^+}\rightarrow x_i$$
in the weak topology of $X$. Moreover, it is clear that
$$\left\{\left\Vert \sum_{i=1}^n \lambda_i(x_i-x_s^i) \right\Vert \right\}\rightarrow 1+\left\Vert \sum_{i=1}^n \lambda_i x_i\right\Vert$$
from triangle inequality.

(2)$\Rightarrow$(1). Is similar to Proposition
\ref{caraDD2Predes}.

\end{proof}

Now we can establish a dual version for the result above.

\begin{proposition}\label{caraDSD2Pdualredes}

Let $X$ be a dual Banach space. The following assertions are
equivalent:

\begin{enumerate}
\item $X$ has the $w^*$-DSD2P. \item For each $x_1,\ldots, x_n\in
B_X$ and each $\lambda_1,\ldots, \lambda_n\in\mathbb R^+$ such
that $\sum_{i=1}^n \lambda_i=1$ it follows that, for each
$i\in\{1,\ldots, n\}$, there exists $\{x_s^i\}_{s\in S}$ a net in
$B_X$ convergent to $x_i$ in the weak-star topology of $X$ such
that
$$\left\{\left\Vert \sum_{i=1}^n \lambda_i(x_i-x_s^i) \right\Vert \right\}\rightarrow 1+\left\Vert \sum_{i=1}^n \lambda_i x_i\right\Vert.$$
\end{enumerate}
\end{proposition}

As we have checked, DLD2P and DD2P have strong links with the rank
one projections. This fact turns out to be also true for the DSD2P
when we consider finite-rank projections.

\begin{proposition}
Let $X$ be a Banach space. Assume that $X$ has the DSD2P. Then for each $p:=\sum_{i=1}^n x_i^*\otimes x_i$ projection we
have
$$\Vert I-p\Vert\geq 1+\left\Vert \sum_{i=1}^n \frac{1}{n} x_i\right\Vert.$$
\end{proposition}

\begin{proof}

Pick $p:=\sum_{i=1}^n x_i^*\otimes x_i$ a finite rank projection
and let $\varepsilon\in\mathbb R^+$.

 Then
$$\sum_{i=1}^n \frac{1}{n} \frac{x_i}{\Vert x_i\Vert}\in \sum_{i=1}^n \frac{1}{n} \left\{y\in B_X\ :\ \begin{array}{c}
\left\vert x_i^*(y)-\frac{1}{\Vert x_i\Vert}\right\vert \Vert x_i\Vert<\frac{\varepsilon}{4}\\
\vert x_j(y)\vert\Vert x_j\Vert<\frac{\varepsilon}{4}\ \ \forall
j\neq i.
\end{array} \right\}$$
As $X$ has the DSD2P then, for each $i\in\{1,\ldots, n\}$, there
exists $y_i\in \left\{y\in B_X\ :\ \begin{array}{c}
\left\vert x_i^*(y)-\frac{1}{\Vert x_i\Vert}\right\vert \Vert x_i\Vert<\frac{\varepsilon}{4}\\
\vert x_j(y)\vert\Vert x_j\Vert<\frac{\varepsilon}{4}\ \ \forall
j\neq i.
\end{array} \right\}$ such that
$$\left\Vert \sum_{i=1}^n \frac{1}{n}\left(y_i-\frac{x_i}{\Vert x_i\Vert}\right)\right\Vert>1+\left\Vert \sum_{i=1}^n \frac{1}{n}\frac{x_i}{\Vert x_i\Vert}\right\Vert-\frac{\varepsilon}{4}.$$
Then
$$\Vert I-p\Vert\geq \left\Vert \sum_{i=1}^n \frac{1}{n} y_i-p\left(\sum_{i=1}^n \frac{1}{n} y_i\right)\right\Vert\geq \left\Vert \sum_{i=1}^n \frac{1}{n} \left(y_i-\frac{x_i}{\Vert x_i\Vert}\right)\right\Vert$$
$$-\left\Vert \sum_{i=1}^n \frac{1}{n} \left(\frac{x_i}{\Vert x_i\Vert}-p\left(\sum_{i=1}^n \frac{1}{n} y_i\right)\right)\right\Vert.$$
Now
$$\left\Vert \sum_{i=1}^n \frac{1}{n} \left(\frac{x_i}{\Vert x_i\Vert}-p\left(\sum_{i=1}^n \frac{1}{n} y_i\right)\right)\right\Vert\leq \sum_{i=1}^n \frac{1}{n}\left\Vert \frac{x_i}{\Vert x_i\Vert}-\sum_{j=1}^n  x_j^*(y_i)x_j \right\Vert\leq$$
$$\sum_{i=1}^n \frac{1}{n} \left( \left\vert \frac{1}{\Vert x_i\Vert}-x_i^*(x_i)\right\vert \Vert x_i\Vert+\sum_{j\neq i} \vert x_j^*(y_i)\vert\Vert x_j\Vert\right)<\frac{\varepsilon}{4}.$$
Thus
$$\Vert I-p\Vert\geq 1+\left\Vert \sum_{i=1}^n \frac{1}{n} x_i\right\Vert.$$
\end{proof}

Now, as we have done in Theorem \ref{DD2Pestableprodu} for the
DD2P, we will focus on analysing the DSD2P in the $\ell_p$ sum of
two Banach spaces. As every Banach space enjoying to have the
DSD2P has the strong diameter two property, we conclude that the
$\ell_p$ sum of two Banach spaces does not have the DSD2P whenever
$1<p<\infty$ \cite[Theorem 3.2]{abl}. Nevertheless, we will prove
that, as well as happens with Daugavet spaces, DSD2P has a nice
behavior in the case $p=\infty$. We shall begin proving the
following

\begin{proposition}

Let $X,Y$ be Banach spaces and assume that $X\oplus_p Y$ has the
DSD2P for $p\in \{1,\infty\}$. Then $X$ and $Y$ enjoy to have the DSD2P.

\end{proposition}

\begin{proof}
In order to prove the Proposition, assume that $X$ does not
satisfy the DSD2P. Then there exists $C:=\sum_{i=1}^n \lambda_i
\bigcap\limits_{j=1}^{n_i}S(B_X,x_{ij}^*,\eta_{ij})$ a convex
combination of non-empty relatively weakly open subsets of $B_X$,
an element $\sum_{i=1}^n \lambda_i x_i\in C$ and
$\varepsilon\in\mathbb R^+$ satisfying that
\begin{equation}\label{condinecefallaDSD2P}
\left\Vert \sum_{i=1}^n \lambda_i(x_i-y_i)\right\Vert\leq
1+\left\Vert \sum_{i=1}^n \lambda_i x_i\right\Vert-\varepsilon\
\forall\sum_{i=1}^n \lambda_i y_i\in C.
\end{equation}
Obviously we will assume the non-trivial case (i.e. $\sum_{i=1}^n
\lambda_i x_i\neq 0$), so we can assume, taking
$\delta<\varepsilon$ if necessary, that $\left\Vert \sum_{i=1}^n
\lambda_i x_i\right\Vert-\varepsilon\geq 0$.

Using Lemma \ref{lemakadets} as much times as necessary we can
assume that each number $\eta_{ij}$ are equal (say $\eta$) and
that  $\eta<\frac{\varepsilon}{2}\ \forall i\in\{1,\ldots, n\}$.
Define
$$\mathcal C:=\sum_{i=1}^n \lambda_i\bigcap\limits_{j=1}^{n_i}
S(B_{X\oplus_p Y},(x_{ij}^*,0),\eta). $$
If $p=1$ we have from \cite[Theorem 3.1, equation (3.1)]{abl} that
\begin{equation}\label{condinece1suma}
\mathcal C\subseteq C\times \eta B_Y.
\end{equation}
So consider $\sum_{i=1}^n \lambda_i (x_i,0)\in \mathcal C$ and
pick $\sum_{i=1}^n\lambda_i (x_i',y_i')\in \mathcal C$. Then
$$\left\Vert \sum_{i=1}^n \lambda_i ((x_i,0)-(x_i',y_i')) \right\Vert=\left\Vert \sum_{i=1}^n \lambda_i (x_i-x_i')\right\Vert+\left\Vert \sum_{i=1}^n \lambda_i y_i'\right\Vert.$$
Now on the one hand, from (\ref{condinecefallaDSD2P}), we have the
inequality
$$\left\Vert \sum_{i=1}^n \lambda_i (x_i-x_i')\right\Vert\leq 1+\left\Vert \sum_{i=1}^n \lambda_i x_i\right\Vert-\varepsilon.$$
On the other hand we have from (\ref{condinece1suma}) the
following
$$\left\Vert \sum_{i=1}^n \lambda_i y_i'\right\Vert<\eta.$$
So combining both previous inequalities and keeping in mind that
$\eta<\frac{\varepsilon}{2}$ we conclude
$$\left\Vert \sum_{i=1}^n \lambda_i ((x_i,0)-(x_i',y_i')) \right\Vert\leq 1+\left\Vert \sum_{i=1}^n \lambda_i x_i\right\Vert-\frac{\varepsilon}{2}.$$
From the arbitrariness of $\sum_{i=1}^n\lambda_i (x_i',y_i')\in
\mathcal C$ we conclude that $X\oplus_1 Y$ fails the DSD2P, so we
are done in the case $p=1$.

The case $p=\infty$ is quite easier than the above one. Indeed,
pick $\sum_{i=1}^n \lambda_i (x_i',y_i')\in\mathcal C$. Then
$$\left\Vert \sum_{i=1}^n \lambda_i ((x_i,0)-(x_i',y_i')) \right\Vert=\max\left\{\left\Vert \sum_{i=1}^n \lambda_i (x_i-x_i')\right\Vert,\left\Vert \sum_{i=1}^n \lambda_i y_i\right\Vert \right\}$$
$$\leq \max\left\{1+\left\Vert \sum_{i=1}^n \lambda_i x_i\right\Vert-\varepsilon,1 \right\}=1+\left\Vert \sum_{i=1}^n \lambda_i x_i\right\Vert-\varepsilon,$$
where the last inequality holds from the assumption
$\left\Vert\sum_{i=1}^n \lambda_i x_i\right\Vert-\varepsilon\geq
0$.

Hence, $X\oplus_\infty Y$ does not have the DSD2P, so we are
done.\vspace{0.5cm}

\end{proof}

Now we shall establish the converse of the result above for
$p=\infty$.

\begin{theorem}\label{DSD2Pestainfinsum}

Let $X,Y$ be a Banach spaces. If $X$ and $Y$ have the DSD2P so does $Z:=X\oplus_\infty Y$.

\end{theorem}

\begin{proof}
Pick $n\in\mathbb N$, $(x_1,y_1),\ldots, (x_n,y_n)\in B_Z$ and
$\lambda_1,\ldots, \lambda_n\in \mathbb R^+$ such that
$\sum_{i=1}^n \lambda_i=1$. In order to prove that $Z$ has the
DSD2P we shall use Proposition \ref{caraDSD2Predes}. As
$$\left\Vert \sum_{i=1}^n \lambda_i(x_i,y_i)\right\Vert_\infty=
\max\left\{\left\Vert \sum_{i=1}^n \lambda_i
x_i\right\Vert,\left\Vert \sum_{i=1}^n \lambda_i y_i\right\Vert
\right\},$$
then either $\left\Vert \sum_{i=1}^n
\lambda_i(x_i,y_i)\right\Vert_\infty= \left\Vert \sum_{i=1}^n
\lambda_i x_i\right\Vert$ or $\left\Vert \sum_{i=1}^n
\lambda_i(x_i,y_i)\right\Vert_\infty= \left\Vert \sum_{i=1}^n
\lambda_i y_i\right\Vert$. We shall assume, with no loss of
generality, that $\left\Vert \sum_{i=1}^n
\lambda_i(x_i,y_i)\right\Vert_\infty= \left\Vert \sum_{i=1}^n
\lambda_i x_i\right\Vert$. Now, as $X$ has the DSD2P, we have from
Proposition \ref{caraDSD2Predes} that, for each $i\in\{1,\ldots,
n\}$, there exists $\{x_s^i\}$ a net weakly convergent to $x_i$
such that
$$\left\Vert \sum_{i=1}^n \lambda_i (x_i-x_s^i)\right\Vert\rightarrow 1+\left\Vert\sum_{i=1}^n \lambda_i x_i\right\Vert.$$
Now we have that $\{(x_s^i,y_i)\}\rightarrow (x_i,y_i)$ in the
weak topology of $Z$ for each $i\in\{1,\ldots, n\}$. Moreover,
from the definition of the norm on $Z$, we have that
$(x_s^i,y_i)\in B_Z$ for each $i\in\{1,\ldots, n\}$ and for each
$s$. Finally, given $s$ one has
$$1+\left\Vert \sum_{i=1}^n \lambda_i (x_i,y_i)\right\Vert_\infty\geq \left\Vert \sum_{i=1}^n \lambda_i((x_i,y_i)-(x_s^i,y_i))
\right\Vert_\infty$$
$$\geq \left\Vert \sum_{i=1}^n \lambda_i (x_i-x_s^i)\right\Vert\rightarrow 1+\left\Vert\sum_{i=1}^n \lambda_i x_i\right\Vert=1+\left\Vert \sum_{i=1}^n \lambda_i (x_i,y_i)\right\Vert_\infty. $$
So $\left\Vert \sum_{i=1}^n \lambda_i((x_i,y_i)-(x_s^i,y_i))
\right\Vert_\infty\rightarrow 1+\left\Vert \sum_{i=1}^n \lambda_i
(x_i,y_i)\right\Vert_\infty$. Consequently, $Z$ has the DSD2P
applying Proposition \ref{caraDSD2Predes}, so we are done. \end{proof}

Finally we will analyze the inheritance of DSD2P to subspaces.
Again in \cite{blr} it is proved that given $X$ a Banach space
with the SD2P and $Y\subseteq X$ a closed subspace such that
$X/Y$ is strongly regular, then $Y$ has the SD2P. Following
similar ideas we have the following

\begin{theorem}\label{pasoDSD2P}
Let $X$ be a Banach space and $Y\subseteq X$ be a closed subspace. If $X$ has the DSD2P and $X/Y$ is strongly regular then $Y$ also
has the DSD2P.
\end{theorem}
\begin{proof}

Let $$C:=\sum_{i=1}^n \lambda_i W_i=\sum_{i=1}^n \lambda_i \left\{
y\in B_Y\ :\ \vert y_{ij}^*(y-y_0^i)\vert<\eta_{ij}\ \ 1\leq j\leq
n_i \right\}$$ be a convex combination of non-empty relatively weakly
open subsets of $B_Y$, where $y_{ij}^*\in B_{Y^*}$ for each
$i\in\{1,\ldots, n\}, j\in\{1,\ldots, n_i\}$ and $y_0^i\in Y$ for
each $i\in\{1,\ldots, n\}$. Pick $\sum_{i=1}^n \lambda_i x_i\in
C$, $\varepsilon\in\mathbb R^+$ and let us prove that there exists
$\sum_{i=1}^n \lambda_i y_i\in C$ such that
$$\left\Vert \sum_{i=1}^n \lambda_i(x_i-y_i)\right\Vert>1+\left\Vert\sum_{i=1}^n \lambda_i x_i\right\Vert-\varepsilon.$$
To this aim pick $0<\delta$ such that
\begin{equation}\label{condidelta}
\vert
y_{ij}^*(x_i-y_0^i)\vert+\frac{\delta}{32}+\left(1+\frac{\delta}{32}\right)\left(1-\frac{1}{1+\frac{\delta}{32}}
\right)<\eta_{ij}\ \forall i\in\{1,\ldots, n\}
\end{equation}
and
\begin{equation}\label{condideltaepsilon}
\frac{\delta}{16}+2\left(1+\frac{\delta}{32}\right)\left(1-\frac{1}{1+\frac{\delta}{32}}
\right)+\frac{\delta}{4}<\varepsilon.
\end{equation}

Let $\pi:X\longrightarrow X/Y$ the quotient map. We have no
loss of generality, by Hahn-Banach theorem, if we assume that
$y_{ij}^*\in B_{X^*}$ for each  $i\in\{1,\ldots, n\},
j\in\{1,\ldots, n_i\}$. Consider $\mathcal W_i$ the non-empty
relatively weakly open subset of $B_X$ defined by $W_i$ for each
$i\in\{1,\ldots, n\}$.

For each $i\in\{1,\ldots, n\}$ consider $A_i:=\pi(\mathcal W_i)$,
which is a convex subset of $B_{X/Y}$ containing to zero. By
\cite[Proposition III.6]{ggms} then $\overline{A_i}$ is equal to
the closure of the set of its strongly regular  points. As a
consequence, for each  $i\in\{1,\ldots, n\}$, there exists $a_i$ a
strongly regular point of $\overline{A_i}$ such that
\begin{equation}\label{normafuertexp}
\left\Vert a_i\right\Vert<\frac{\delta}{32}.
\end{equation}
For every $i\in\{1,\ldots,n\}$ we can find $m_i\in\mathbb N,
\mu_1^i,\ldots, \mu_{m_i}^i\in ]0,1]$ such that $\sum_{j=1}^{m_i}
\mu_j^i=1$ and $(a_1^i)^*,\ldots, (a_{m_i}^i)^*\in
S_{(X/Y)^*},\alpha_j^i\in\mathbb R^+$ satisfying that
$$a_i\in \sum_{j=1}^{m_i}\mu_j^i (S(B_{X/Y},(a_j^i)^*,\alpha_j^i)\cap \overline{A_i})$$
and
\begin{equation}\label{diaslicescociente}
diam\left(\sum_{j=1}^{m_i} \mu_j^i
(S(B_{X/Y},(a_j^i)^*,\alpha_j^i)\cap
\overline{A_i})\right)<\frac{\delta}{32}.
\end{equation}
It is clear that, for $i\in\{1,\ldots, n\}$ and $j\in\{1,\ldots,
m_i\}$, one has
$$S(B_{X/Y},(a_j^i)^*,\alpha_j^i)\cap A_i\neq \emptyset\Rightarrow S(B_X,\pi^*((a_j^i)^*),\alpha_j^i)\cap \mathcal W_i\neq \emptyset.$$
Check that we can not still apply the hypothesis because we do not
know whether $\sum_{i=1}^n \lambda_i x_i\in \mathcal
C:=\sum_{i=1}^n \lambda_i \sum_{j=1}^{m_i}\mu_j^i (\mathcal
W_i\cap S(B_X,\pi^*((a_j^i)^*),\alpha_j^i))$. Now, in order to
finish the proof, we need to find points in $\mathcal C$ close
enough to  $\sum_{i=1}^n \lambda_i x_i$. This will be done in the
following

\begin{claim}

We can find, for each $i\in\{1,\ldots, n\}$, an element $z_i\in
B_X$ such that $\sum_{i=1}^n \lambda_i z_i\in \mathcal C$ and that
\begin{equation}\label{elementoscerca}
\left\Vert \sum_{i=1}^n \lambda_i
(x_i-z_i)\right\Vert<\frac{\delta}{32}+\left(1+\frac{\delta}{32}\right)\left(1-\frac{1}{1+\frac{\delta}{32}}\right).
\end{equation}

\end{claim}

\begin{proof}
Pick $i\in\{1,\ldots, n\}$. As $\Vert \pi(x_i)-a_i\Vert=\Vert
a_i\Vert<\frac{\delta}{32}$ we can find $z_i\in X$ such that
$\pi(z_i)=a_i$ and such that
\begin{equation}\label{reprepuntoregular}
\Vert x_i-z_i\Vert<\frac{\delta}{32}.
\end{equation}
Now $a_i\in \sum_{j=1}^{m_i}\mu_j^i
(S(B_{X/Y},(a_j^i)^*,\alpha_j^i)\cap \overline{A_i})$ so, for each
$j\in\{1,\ldots, m_i\}$, we can find $b_{ij}\in
S(B_{X/Y},(a_j^i)^*,\alpha_j^i)\cap \overline{A_i}$ such that
$a_i=\sum_{j=1}^{m_i}\mu_j^i b_{ij}$. For each $j\in\{1,\ldots,
m_i\}$ we can find, considering a perturbation argument if
necessary, an element $z_{ij}\in B_X$ such that
$\pi(z_{ij})=b_{ij}$. Finally, as $\pi(z_i)-\sum_{j=1}^{m_i}
\mu_j^i\pi(z_{ij})=0$ we can find, by definition of the norm on
$X/Y$, an element $y_i\in Y$ such that
\begin{equation}\label{defiy}
z_i=\sum_{j=1}^{m_i}\mu_j^iz_{ij}+y_i,
\end{equation}
and
\begin{equation}\label{normay}
\Vert y_i\Vert<\frac{\delta}{32}.
\end{equation}
We shall prove that $\sum_{i=1}^n \lambda_i
\frac{z_i}{1+\frac{\delta}{32}}$ works. First of all we have
$$\left\Vert\sum_{i=1}^n \lambda_i z_i\right\Vert\leq \left\Vert\sum_{i=1}^n \lambda_i x_i\right\Vert+\sum_{i=1}^n \lambda_i \Vert x_i-z_i\Vert\mathop{<}\limits^
{\mbox{{\small{(\ref {reprepuntoregular})}}}}
1+\frac{\delta}{32}.$$
So $\frac{z_i}{1+\frac{\delta}{32}}\in B_X$ for each
$i\in\{1,\ldots, n\}$.

Moreover, given $i\in\{1,\ldots, n\}, j\in\{1,\ldots, n_i\}$, one
has
$$\left\vert y_{ij}^*\left(\frac{z_i}{1+\frac{\delta}{32}}-y_0^i \right) \right\vert\leq \vert y_{ij}^*(x_i-y_0^i)\vert+\Vert x_i-z_i\Vert+\left\Vert z_i-\frac{z_i}{1+\frac{\delta}{32}}\right\Vert$$
$$\mathop{<}
\limits^{\mbox{{\small{ (\ref{reprepuntoregular})}}}}  \vert
y_{ij}^*(x_i-y_0^i)\vert+\frac{\delta}{32}+\left(1+\frac{\delta}{32}\right)\left(1-\frac{1}{1+\frac{\delta}{32}}
\right)\mathop{<} \limits^{\mbox{ {\small{(\ref{condidelta})}}}}
\eta_{ij}.$$
Finally, pick $i\in\{1,\ldots, n\}$ and $j\in\{1,\ldots, m_i\}$.
Then by (\ref{defiy}) one has
$$z_i=\sum_{j=1}^{m_i}\mu_j^i(z_{ij}+y_i).$$
In addition
$$\pi^*(a_{ij}^*)(z_{ij}+y_i)=a_{ij}^*(b_{ij})+
a_{ij}^*(\pi(y_i)).$$
On the one hand, as $y_i\in Y$ then $\pi(y_i)=0$. On the other
hand $a_{ij}^*(b_{ij})>1-\alpha_j^i$. Now, up to consider a
smaller positive number in (\ref{reprepuntoregular}) (check that
the choice of $b_{ij}$ does not depend on the one of $z_1,\ldots,
z_n$), we can assume that $a_{ij}^*(b_{ij})>(1-\alpha_j^i)
\left(1+\frac{\delta}{32}\right)$, so $\sum_{i=1}^n \lambda_i
\frac{z_i}{1+\frac{\delta}{32}}\in\mathcal C$. Now  the claim
follows just considering $\frac{z_i}{1+\frac{\delta}{32}}$ instead
of $z_i$.

\end{proof}

Now, as $X$ has the DSD2P, we can find $\sum_{i=1}^n \lambda_i
z_i'\in \mathcal C$ such that
\begin{equation}\label{aplicahipoDSD2P}
\left\Vert\sum_{i=1}^n \lambda_i
(z_i-z_i')\right\Vert>1+\left\Vert\sum_{i=1}^n \lambda_i
z_i\right\Vert-\frac{\delta}{32}.
\end{equation}
Given $i\in\{1,\ldots, n\}$ we have that
$$\pi(z_i')\in \sum_{j=1}^{m_i}\mu_j^i (S(B_{X/Y},(a_j^i)^*,\alpha_j^i)\cap \overline{A_i})$$
$$\Rightarrow \Vert \pi(z_i')\Vert\leq \Vert a_i\Vert+ diam\left(\sum_{j=1}^{m_i}\mu_j^i (S(B_{X/Y},(a_j^i)^*,\alpha_j^i)\cap \overline{A_i}) \right)\mathop{<}\limits^{\mbox
{{\small{(\ref{normafuertexp})(\ref{diaslicescociente})}}}}
\frac{\delta}{16}. $$
Now, as it is done in Proposition \ref{pasoDD2P}, we can find
$y_i\in B_Y$ such that
\begin{equation}\label{elementoseny}
\Vert y_i-z_i'\Vert<\frac{\delta}{4}.
\end{equation}
Now, on the one hand, given $i\in\{1,\ldots, n\}$ and
$j\in\{1,\ldots, n_i\}$, one has
$$\vert y_{ij}^*(y_i-y_0)\vert\leq \vert y_{ij}^*(z_i'-y_0)\vert+\vert y_{ij}^*(y_i-z_i)\vert <\eta_{ij}+\frac{\delta}{4}.$$
On the other hand,
$$\left\Vert\sum_{i=1}^n \lambda_i(x_i-y_i)\right\Vert\geq \left\Vert\sum_{i=1}^n \lambda_i (z_i-z_i')\right\Vert-\left\Vert\sum_{i=1}^n \lambda_i (x_i-z_i)\right\Vert-\left\Vert\sum_{i=1}^n \lambda_i (y_i-z_i')\right\Vert$$
$$\mathop{>}\limits^{\mbox{{\small{
(\ref{aplicahipoDSD2P})}}}}1+ \left\Vert\sum_{i=1}^n \lambda_i
z_i\right\Vert-\frac{\delta}{32}-\sum_{i=1}^n \lambda_i\Vert
y_i-z_i'\Vert-\left\Vert\sum_{i=1}^n \lambda_i
(z_i-x_i)\right\Vert$$
$$\mathop{>}\limits^{\mbox{{\small{
(\ref{elementoscerca})(\ref{elementoseny})}}}} 1+
\left\Vert\sum_{i=1}^n \lambda_i
z_i\right\Vert-\frac{\delta}{32}-\frac{\delta}{4}-\frac{\delta}{32}-\left(1+\frac{\delta}{32}\right)\left(1-\frac{1}{1+\frac{\delta}{32}}\right)$$
$$\mathop{>}\limits^{\mbox{{\small{
(\ref{elementoscerca})}}}}1 +\left\Vert\sum_{i=1}^n \lambda_i
x_i\right\Vert-\frac{\delta}{16}-\frac{\delta}{4}-2\left(1+\frac{\delta}{32}\right)\left(1-\frac{1}{1+\frac{\delta}{32}}\right)$$
$$\mathop{>}\limits^{\mbox
{{\small{(\ref{condideltaepsilon})}}}}1+ \left\Vert\sum_{i=1}^n
\lambda_i x_i\right\Vert-\varepsilon.$$
From the arbitrariness of $\varepsilon$ we conclude that $Y$ has
the DSD2P by a perturbation argument similar to the one done in
Proposition \ref{pasoDD2P}.\end{proof}

Now we have a weak-star version of Theorem \ref{pasoDSD2P}.

\begin{corollary}\label{subespaciosw^*DSD2P}

Let $X$ be a Banach space and let $Y\subseteq X$ be a closed
subspace. If $X^*$ has the $w^*$-DSD2P and $Y$ is reflexive then $(X/Y)^*$
has the $w^*$-DSD2P.

\end{corollary}

\begin{proof}
Consider $C:=\sum_{i=1}^n \lambda_i W_i$ a convex combination of
non-empty relatively weakly-star open subsets of $B_{Y^\circ}$ and
pick $\sum_{i=1}^n \lambda_i z_i^*\in C$.

Define $\mathcal W_i$ to be the weak-star open subset of $B_{X^*}$
define by $W_i$ for each $i\in\{1,\ldots, n\}$.

Let $\pi:X^*\longrightarrow X^*/Y^\circ$ the quotient map and
define $A_i:=\pi(\mathcal W_i)$.

As $X^*/Y^\circ=Y^*$ is reflexive, then $X^*/Y^\circ$ is strongly
regular, so we can find, for each $i\in\{1,\ldots, n\}$, $a_i$ a
point of strong regularity point of $\overline{A_i}$ whose norm is
as close to zero as desired. Given $i\in\{1,\ldots, n\}$, as $a_i$
is a point of strong regularity, we can find convex combination of
slices containing $a_i$ and whose diameter is as small as wanted.
In addition, because of reflexivity of $X^*/Y^\circ$, convex
combination of slices are indeed convex combination of weak-star
slices, so we can actually find convex combination of weak-star
slices containig to $a_i$ and whose diameter is a closed to zero
as desired for each $i\in\{1,\ldots, n\}$.

Using the previous ideas, the result can be concluded following
word by word the proof of Theorem \ref{pasoDSD2P}.
\end{proof}

Now we shall prove the inheritance of the DSD2P to almost isometric ideals.

\begin{proposition}\label{aidsd2p}

Let $X$ be a Banach space and let $Y\subseteq X$ a closed almost
isometric ideal. If $X$ has the DSD2P, so does $Y$.

\end{proposition}

\begin{proof}

Pick $C:=\sum_{i=1}^n \lambda_i \bigcap\limits_{j_1}^{n_i}
S(B_Y,y_{ij}^*,\alpha_{ij})$ a convex combination of non-empty
relatively weakly open subsets of $B_Y$, choose $\sum_{i=1}^n
\lambda_i y_i\in C$ and pick $\varepsilon>0$. Our aim is to find
$\sum_{i=1}^n \lambda_i z_i\in C$ such that $\left\Vert
\sum_{i=1}^n \lambda_i (y_i-z_i)\right\Vert>1+\left\Vert
\sum_{i=1}^n \lambda_i y_i\right\Vert-\varepsilon$. Assume, with
no loss of generality, that $\max\limits_{1\leq i\leq
n}\max\limits_{1\leq j\leq n_i}y_{ij}^*(y_i)<1$.

Choose $\mu_0>0$ such that
\begin{equation}\label{aiDSD2Pcondimufunc}
0<\mu<\mu_0\Rightarrow \frac{y_{ij}^*(y_i)}{1+\mu}>1-\alpha_{ij}\
\forall i\in\{1,\ldots, n\}, j\in\{1,\ldots, n_i\}.
\end{equation}
and
\begin{equation}\label{aiDSD2Pcondimuestima}
0<\mu<\mu_0\Rightarrow
\frac{\frac{1}{1+\mu}\left(1+\left\Vert\sum_{i=1}^n \lambda_i
y_i\right\Vert-\mu \right)-\mu}{1+\mu}>1+\left\Vert\sum_{i=1}^n
\lambda_i y_i\right\Vert-\varepsilon
\end{equation}
Now consider $0<\mu<\mu_0$ and  $\varphi:Y^*\longrightarrow X^*$ a
Hahn-Banach operator satisfying the properties described in
Theorem \ref{propiai}. Define
$$\widehat{C}:=\sum_{i=1}^n \lambda_i \bigcap\limits_{j=1}^{n_i} S(B_X,\varphi(y_{ij}^*),1-y_{ij}^*(y_i)).$$
As $X$ has the DSD2P and clearly $\sum_{i=1}^n \lambda_i y_i\in
\widehat{C}$ we can conclude the existence of an element
$\sum_{i=1}^n \lambda_i x_i\in \widehat{C}$ such that
\begin{equation}\label{DSD2Paidistaelem}
\left\Vert \sum_{i=1}^n \lambda_i
(y_i-x_i)\right\Vert>1+\left\Vert\sum_{i=1}^n \lambda_i
y_i\right\Vert-\mu.
\end{equation}
Now for $\mu$, $E:=span\{x_1,\ldots, x_n, y_1,\ldots,
y_n\}\subseteq E$ and $F:=span\{y_{ij}^*\ /\ i\in\{1,\ldots, n\},
j\in \{1,\ldots, n_i\}\}\subseteq Y^*$ consider $T$ the operator
satisfying the properties described in Theorem \ref{propiai}.
Given $i\in\{1,\ldots, n\}$ one has
$$\Vert T(x_i)\Vert\leq (1+\mu)\Vert x_i\Vert\leq 1+\mu.$$
So, if we define $z:=\sum_{i=1}^n \lambda_i \frac{T(x_i)}{1+\mu}$,
it is clear that $z\in B_Y$. We will prove that indeed $z\in C$.
To this aim, pick $i\in\{1,\ldots, n\}$ and $j\in\{1,\ldots,
n_i\}$. Hence
$$y_{ij}^*\left(\frac{T(x_i)}{1+\mu} \right)=\frac{y_{ij}^*(T(x_i))}{1+\mu}=\frac{\varphi(y_{ij}^*)(x_i)}{1+\mu}>\frac{1-(1-y_{ij}^*(y_i))}{1+\mu}=\frac{y_{ij}^*(y_i)}{1+\mu}$$
$$\mathop{>}\limits^
{\mbox{(\ref{aiDSD2Pcondimufunc})}} 1-\alpha_{ij}.$$
Thus $z\in C$. Finally, we have that
$$\left\Vert \sum_{i=1}^n \lambda_i y_i-z\right\Vert=\left\Vert \sum_{i=1}^n \lambda_i \left(y_i-\frac{T(x_i)}{1+\mu} \right) \right\Vert=\frac{\left\Vert \sum_{i=1}^n \lambda_i(y_i-T(x_i)+\mu y_i)\right\Vert}{1+\mu}$$
$$\geq \frac{\left\Vert T\left( \sum_{i=1}^n \lambda_i (y_i-x_i)\right)\right\Vert-\mu\sum_{i=1}^n \lambda_i \Vert y_i\Vert}{1+\mu}>\frac{\frac{1}{1+\mu}\left\Vert \sum_{i=1}^n \lambda_i(y_i-x_i)\right\Vert-\mu}{1+\mu}$$
$$\mathop{>}\limits^{\mbox
{(\ref{DSD2Paidistaelem})}}\frac{\frac{1}{1+\mu}\left(1+\left\Vert\sum_{i=1}^n
\lambda_i y_i\right\Vert-\mu
\right)-\mu}{1+\mu}\mathop{>}\limits^{\mbox{(\ref
{aiDSD2Pcondimuestima})}}  1+\left\Vert\sum_{i=1}^n \lambda_i
y_i\right\Vert-\varepsilon.$$
As $\varepsilon$ was arbitrary we conclude that $Y$ has the DSD2P,
so we are done.

\end{proof}

\section{Some remarks and open questions.}

\par
\bigskip

Let us consider the following diagram
$$\begin{array}{ccccccc}
DP & \mathop{\Longrightarrow}\limits^{(1)} & DSD2P & \mathop{\Longrightarrow}\limits^{(2)} & DD2P & \mathop{\Longrightarrow}\limits^{(3)} & DLD2P\\
\ & \ & \Downarrow (4) & \ & \Downarrow (5) & \ & \Downarrow (6)\\
\ & \ & w^*-DSD2P & \mathop{\Longrightarrow}\limits^{(7)}&
w^*-DD2P & \mathop{\Longrightarrow}\limits^{(8)}& w^*-DLD2P
\end{array}$$
where the last row only make sense in dual Banach spaces. By
Example \ref{ejeDD2Pnodauga} or Theorem \ref{DD2Pestableprodu},
neither the converse implication of (2) nor the one of (7) holds.
In addition, there are Banach spaces with the Daugavet property
whose dual unit ball have denting points (e.g. $\mathcal
C([0,1])$). Consequently, the converse of (4),(5) or (6) is not
true.

However, it remains open the following

\begin{question}
Does the converse of (1),(3) or (8) hold?
\end{question}

It is known that a Banach space $X$ has the DLD2P if, and only if,
$X^*$ has the $w^*$-DLD2P. This fact arise two questions.

\begin{question}

Let $X$ be a Banach space. Is it true that $X$ has the DD2P if, and only if, $X^*$ has the
$w^*$-DD2P?

\end{question}

A similar question remains open for the DSD2P.

\begin{question}

Let $X$ be a Banach space. Is it true that $X$ has the DSD2P if, and only if, $X^*$ has the
$w^*$-DSD2P?
\end{question}

Check that a positive answer to the above question would provide,
by Proposition \ref{caraDSD2Pdualredes} and a similar argument to
the one done in Theorem \ref{DSD2Pestainfinsum} to the dual space,
a positive answer to the following
\begin{question}
Let $X,Y$ Banach space. Does $X\oplus_1 Y$ enjoy to have the DSD2P whenever $X$ and $Y$
have the DSD2P?
\end{question}

\end{document}